\documentclass[12pt]{amsart}
\usepackage{graphicx,amssymb}
\usepackage{pb-diagram}
\usepackage[all]{xy}
\usepackage{color}

\DeclareMathOperator{\core}{Core}


\newtheorem{thm}{Theorem}
\newtheorem{cor}[thm]{Corollary}

\newtheorem{prop}[thm]{Proposition}
\theoremstyle{definition}

\theoremstyle{plain}
\newtheorem*{MT}{Main Theorem}

\theoremstyle{remark}
\newtheorem{rem}[thm]{Remark}

\numberwithin{equation}{section}

\makeindex

\begin{document}

\title{The Monodromy group of $pq$-covers}
\author[A. Carocca]{\'Angel Carocca}
\email{angel.carocca@ufrontera.cl}

\author[R. E. Rodr\'{\i}guez]{Rub\'{\i} E. Rodr\'{\i}guez}
\email{rubi.rodriguez@ufrontera.cl}

\address{Departamento de Matem\'atica y Estad\'{\i}stica, Universidad de La Frontera. Temuco, Chile}


\thanks{The authors were partially supported by  Fondecyt Grants 1190991 and 1200608}
\begin{abstract}
In this work we study the monodromy group of  covers $\varphi \circ \psi$ of curves \linebreak $\mathcal{Y}\xrightarrow {\quad {\psi}} \mathcal{X} \xrightarrow {\quad \varphi} \mathbb{P}^{1}$, where $\psi$ is a $q$-fold cyclic \'etale cover and $\varphi$ is a totally ramified $p$-fold cover, with $p$ and $q$ different prime numbers with $p$ odd.

We show that the Galois group $\mathcal{G}$ of the Galois closure $\mathcal{Z}$ of $\varphi \circ \psi$ is of the form 	$ \mathcal{G} = \mathbb{Z}_q^s \rtimes \mathcal{U}$, where $0 \leq s \leq p-1$ and $\mathcal{U}$  is a simple transitive permutation group of degree $p$.  Since the simple transitive permutation group of prime degree $p$ are known, and we construct examples of such covers with these Galois groups, the result is very different from the previously known case when the cover $\varphi$ was assumed to be cyclic, in which case the Galois group is of the form $ \mathcal{G}  = \mathbb{Z}_q^s \rtimes \mathbb{Z}_p$.

Furthermore, we are able to characterize the subgroups $\mathcal{H}$ and $\mathcal{N}$ of $\mathcal{G}$ such that $\mathcal{Y} = \mathcal{Z}/\mathcal{N}$ and $X = \mathcal{Z}/\mathcal{H}$. 
\end{abstract}

\maketitle

\section{Introduction}

Let $\mathcal{X}$ be a compact Riemann surface  of genus $ \; \mathtt{g}_{\mathcal{X}} \; $ and  $  \varphi: \mathcal{X} \rightarrow \mathbb{P}^{1}$  a cover of degree  $d$. In 1874 Klein determined the possible Galois groups for the case of Galois covers with  $ \; \mathtt{g}_{\mathcal{X}} = 0$. In 1926, O. Zariski \cite{Zariski;1978} found  the classification in the case that the  Galois closure  of $  \varphi $ is a Frobenius group and $ \; \mathtt{g}_{\mathcal{X}} =0$.
Since then, many authors have worked on this problem, either imposing special conditions on the cover or fixing $ \; \mathtt{g}_{\mathcal{X}}, \; $     and many results are known; see for instance   \cite{Allcock;Hall;2010, Artebani;Pirola;2005, Guralnick;1995, Guralnick;Shareshian;2007, Magaard;Volklein;2004}.

Now consider $\mathcal{Y}\xrightarrow {\quad {\psi}} \mathcal{X} \xrightarrow {\quad \varphi} \mathbb{P}^{1}$  a sequence of covers of compact Riemann surfaces.
In this direction, a more general problem is studying the Galois closure of the composite cover $\varphi \circ \psi:\mathcal{Y}\rightarrow \mathbb{P}^{1}$.
Some results on this problem, considering special properties of the covers $\varphi$ and $\psi$, can be found for instance in \cite{Arenas;Rojas;2010, Biggers;Fried;1986, CLR, CHR, Carocca;Romero, Diaz;Donagi;1989, Guralnick;Shareshian;2007, Kanev;2006, Vetro;2007}.

The first cases studied were \'etale double  covers $\mathcal{Y}\xrightarrow {\quad {\psi}} \mathcal{X}$ of cyclic covers $\mathcal{X} \xrightarrow {\quad \varphi} \mathbb{P}^{1}$ of prime degree $p$. For $p=2$ Mumford shows in
\cite{m} that $Y \rightarrow \mathbb{P}^1$ is Galois with Galois group the Klein group of order $4$, and the Prym variety $P(\mathcal{Y}/\mathcal{X})$ is isomorphic as a principally polarized abelian variety to a Jacobian or a product of two Jacobians. The case  $p=3$  corresponds to the well known Recillas trigonal construction \cite{sr}, where the Galois group for the Galois closure of an \'etale double cover of a  cyclic trigonal curve is
the alternating group $\mathbf{A}_4$, and the Prym variety $P(\mathcal{Y}/\mathcal{X})$ is isomorphic as a principally polarized abelian variety to the Jacobian of a tetragonal curve.

These results were extended in \cite{CLR} for any prime $p \geq 5$, where it is shown that the Galois group of the Galois closure of an \'etale double cover of a cyclic cover of degree $p$ is of the form $ \mathcal{G} \cong \mathbb{Z}_2^{p-1} \rtimes \mathbb{Z}_p$.

The Galois group of the Galois closure of \'etale cyclic $q$-fold covers  of cyclic $p$-fold covers was found in \cite{CHR}, where $q$ and $p$ are different primes with $p$ odd and $(p,q-1)=1$; the Galois group in this case is $ \mathcal{G} \cong \mathbb{Z}_q^{s} \rtimes \mathbb{Z}_p$, with $2 \leq s \leq p-1$ and $q^s \equiv 1 \textup{ mod } p$.

In this paper we compute the Galois group $ \; \mathcal{G} \; $ of the Galois closure of the composite cover $\varphi \circ \psi $ when  $ {{\psi}} $ is  a $q$-fold cyclic \'etale cover and ${ \varphi}$ is a totally ramified covering of degree $\; p \; $, for any prime numbers $q \ne p$ and $p \geq 3$. As we will see, the answer in this case gives many more possibilities for the Galois group, including non solvable ones.

\begin{MT}  \label{thm:MT}
	Let $p$ and $q$ be different prime numbers with $p$ odd, and consider  the composite cover  $\mathcal{Y}\xrightarrow {\quad {\psi}} \mathcal{X} \xrightarrow {\quad \varphi} \mathbb{P}^{1}$ where $\psi$ is a $q$-fold cyclic \'etale cover and $\varphi$ is a totally ramified $p$-fold cover.
	
	Let $ \; \widetilde{\varphi} \; : \mathcal{Z} \to \mathbb{P}^{1}$ be the Galois closure of the composite cover $\varphi\circ\psi \; $ and  denote by  $\mathcal{G}$  the corresponding Galois group  of $\widetilde{\varphi}$, with $\mathcal{N} \leq \mathcal{H} \leq \mathcal{G}$ the (conjugacy classes of) subgroups
	such that $\mathcal{Z}/\mathcal{N} = \mathcal{Y}$ and $\mathcal{Z}/\mathcal{H} = \mathcal{X}$.  Also denote by $\mathcal{K} := \mathcal{H}_{\mathcal{G}} = \core_{ \mathcal{G}}(\mathcal{H})$.

Then
$$
\mathcal{G} \cong \mathcal{K} \rtimes \mathcal{U} \, ,
$$
where $\mathcal{K} \cong \mathbb{Z}_q^s$ with $0 \leq s \leq p-1$, and $\mathcal{U}$ is a simple transitive permutation group of prime degree $ \; p$.

\medskip

Furthermore,
\begin{enumerate}
	\item[i)] The group $\mathcal{G}$ is solvable if and only if $\mathcal{H}$ is a normal subgroup of $\mathcal{G}$ if and only if $\mathcal{X} \xrightarrow {\quad \varphi} \mathbb{P}^{1}$ is a cyclic $p$-fold cover.
	In this case $\mathcal{U} \cong \mathbb{Z}_p$, $\mathcal{H}= \mathcal{K} \cong \mathbb{Z}_q^s$ with some $s$ such that $1 \leq  s \leq p-1$, and $\mathcal{N}$ is a maximal subgroup of $\mathcal{H}$ with $\mathcal{N}_{\mathcal{G}} = \{1\}$.

\medskip

     In particular, the original cover $\varphi\circ\psi \; $  is Galois if and only if   $\mathcal{N} =\{1\}$  if and only if   
    $$\mathcal{G} = \langle a, b : a^q = b^p =1 , b^{-1}ab = a^r\rangle \cong \mathbb{Z}_q \rtimes \mathbb{Z}_p \ ,
    $$
    with  $p/(q-1)$ and  $1 \ne r$ a primitive $p$-th root of unity in the field of $q$ elements.

\medskip

   \item[ii)] Otherwise,  $\mathcal{K} \cong \mathbb{Z}_q^s$ with $0 \leq s \leq p-1$ and $\mathcal{U}$ is a simple non-abelian group in the following list.

\begin{enumerate}
	\item  $ \; \mathcal{U}  = {\mathbf A}_p \; $;   $\mathcal{H} \cap \mathcal{U} =  {\mathbf A}_{p-1}$; $p \geq 5$;  \\
	\item  $ \; \mathcal{U} = \mathbf{PSL}(2, 11) \; $  with $ \; p = 11 \; $;  $\mathcal{H} \cap \mathcal{U} =  \mathbf{A}_5$;\\
	\item  $ \; \mathcal{U} = {\mathbf M}_{11} \; $ with $ \; p = 11 \; $; $\mathcal{H} \cap \mathcal{U} =  \mathbf{M}_{10}$;\\
	\item  $  \; \mathcal{U} = {\bf M}_{23} \; $  with $ \; p = 23 \; $; $\mathcal{H} \cap \mathcal{U} = \mathbf{M}_{22}$;\\
	\item  $ \; \mathcal{U}  = \mathbf{PSL}(n,  \mathfrak{q}) \; $,  where $ \displaystyle\frac{\mathfrak{q}^n - 1}{\mathfrak{q} - 1}  = p $, $ \; n \; $ is  prime and $\mathfrak{q}$ is a power of a prime.  $\mathcal{H} \cap \mathcal{U}$ is the stabilizer of a point or a hyperplane of $ \; {\mathbb F}_{\mathfrak{q}}^n$.\\
\end{enumerate}

\medskip

In this case $\mathcal{H}  = \mathcal{K} \rtimes (\mathcal{H} \cap \mathcal{U})$ with $\mathcal{H} \cap \mathcal{U}$ a subgroup of index $p$ in $\mathcal{U}$ and
$\mathcal{N}$ is a maximal subgroup of $\mathcal{H}$ with $\mathcal{N}_{\mathcal{G}} = \{1\}$.

\bigskip

Furthermore, if $\mathcal{K} = \{1\}$ (that is, if $s=0$)  or if $ \;  \mathcal{G}  = {\mathcal N} \: \mathcal{U}$ then $\mathcal{U}$ is one of the following groups:

\begin{enumerate}
	\item[(f)]  $ \;\mathcal{U} = {\bf A}_5 \; $ with $p=5, \; q = 3$;\\
	\item[(g)]  $ \; \mathcal{U} = {\bf M}_{11} \; $ with $p=11, \; q = 2$;\\
	\item[(h)]  $ \; \mathcal{U}  = {\bf PSL}(n,  \mathfrak{q}) \; $  with  $ \; p = \displaystyle\frac{\mathfrak{q}^n - 1}{\mathfrak{q} - 1} \; $  where $ \; n \; $  is prime and such that
	\begin{itemize}
		\item if $ \;  \mathfrak{q}  > 2,  \; \; $ then $ \; q \; $ is any prime divisor of  $ \; \mathfrak{q} -1$;
		\item if $ \; \mathfrak{q} = 2 \; \; $ then $ \; n = 3  \; $, $p=7$  and  $ \; q = 2$.
	\end{itemize}
\end{enumerate}

  \end{enumerate}
\end{MT}

\medskip

This result will be proved in several steps as follows: in Section 2 we fix the notation and recall general results on group actions on curves and some properties of the Galois closure of covers of curves that will be used later. In Section 3 we establish the structure of the Galois group $\mathcal{G}$ of the Galois closure of our cover $\varphi \circ \psi$; equivalently, of the Monodromy representation of the cover $\varphi\circ\psi: \mathcal{Y}\rightarrow \mathcal{X} \rightarrow \mathbb{P}^{1}$,  showing that $\mathcal{G} \cong \mathcal{K} \rtimes \mathcal{U}$, where $\mathcal{K} \cong \mathbb{Z}_q^s$ with $0 \leq s \leq p-1$, and $\mathcal{U}$ is a simple transitive permutation group of prime degree $ \; p$.  In Section 4 we study the case  when $\mathcal{X} \rightarrow \mathbb{P}^{1}$ is a cyclic cover, and show that this is equivalent to $\mathcal{G}$ being solvable and also equivalent to $\mathcal{U} \cong \mathbb{Z}_p$. In Section 5 we study the non-solvable case, showing part ii) in the Theorem. Finally, in Section 6 we give several examples to illustrate.

\section{Preliminaries}\label{sec:0}
\subsection{Group actions on compact Riemann Surfaces}\label{sec:0.1}
In order to fix the notation, we start by recalling some basic properties of group actions on compact Riemann surfaces.

Let $\; \mathcal{X} \;$ be a compact Riemann surface of genus $ \; \mathtt{g}_{\mathcal{X}} \; $ and $G$ a finite group acting
on $\; \mathcal{X}.
 \;$ The quotient projection
$\;\mathcal{X} \rightarrow \mathcal{X}_{G}\;$ is a branched cover, which may be partially characterized
by a vector of numbers $\;(\gamma ; m_{1}, \cdots , m_{r})\;$ where $\;\gamma \; $ is the genus of $\; \mathcal{X}_{G} := \mathcal{X}/{G}, $ the integer  $ \; 0 \leq r \leq 2 \mathtt{g}_{\mathcal{X}}  + 2 \;$ is the number of branch points of the cover and the integers  $ \; m_{j} \; $
are the orders of (representatives of the conjugacy classes of) the cyclic subgroups $ \; G_j \; $ of $ \; G \; $ which fix points on $ \; \mathcal{X}$. We call $(\gamma ; m_{1}, \cdots , m_{r})$ the \textit{branching data} of $G$ on $ \; \mathcal{X}$. These numbers satisfy the Riemann-Hurwitz equation
\begin{equation}\label{rh}
\frac{2(\mathtt{g}_{\mathcal{X}} -1)}{\vert G\vert }=2(\gamma-1)+\sum_{j=1}^{r}\left(1-\frac{1}{m_{j}}\right) .
\end{equation}
\noindent A $(2\gamma+r)-$tuple $\left(a_{1},\cdots, a_{\gamma},b_{1},\cdots,b_{\gamma}, c_{1},\cdots,c_{r}\right)$ of elements of $ \; G \; $ is called a \textit{generating vector of type } $(\gamma;m_{1},\cdots,m_{r})$  if
$$\label{gvector}
 G = \left\langle a_{1},\cdots, a_{\gamma},b_{1},\cdots,b_{\gamma}, c_{1},\cdots,c_{r} \; \; / \;  \prod_{i=1}^{\gamma}[a_{i},b_{i}]\prod_{j=1}^{r}c_{j} = 1 \; ,  \; \vert c_{j}\vert  = m_{j} \;  \mbox {for} \; j =1,...,r \; , \; {\mathcal R}  \right\rangle\\
$$
where $[a_i,b_i]=a_ib_ia_i^{-1}b_i^{-1} \; $ and $ \; {\mathcal R} \; $ is a set of appropriate relations on  the elements $ \; \{ a_{1},\cdots, a_{\gamma},b_{1},\cdots,b_{\gamma}, c_{1},\cdots,c_{r}\}.$
\vspace{2mm}\\
Riemann's Existence Theorem then tells us that (see \cite{Broughton;1991})

\begin{thm}  \label{thm:geomgens}
The group $G$ acts on a surface $\; \mathcal{X} \; $ of genus $\; \mathtt{g}_{\mathcal{X}}$ with branching data $(\gamma;m_{1},\cdots,m_{r})$ if and only if $G$ has a generating vector of type $(\gamma;m_{1},\cdots,m_{r})$ satisfying the Riemann-Hurwitz formula \eqref{rh}.
\end{thm}

We will use this result in constructing examples in Section  \ref{S:ns}.

\subsection{Some properties of the Galois closure of   covers}\label{sec:0.2}
In this section we recall some properties of the Galois group of the Galois closure of a cover of curves.
\vspace{2mm}\\
Let $\mathcal{X} \; $ be a compact Riemann surface  and  $ \; \varphi: \mathcal{X}\rightarrow \mathbb{P}^{1}$  a cover of degree  $p. $ The \textit{Galois closure of $ \; \varphi \; $} is a Galois cover $ \; \widehat{\varphi}:\widehat{\mathcal{X}}\rightarrow \mathbb{P}^{1} \; $ of smallest possible degree such that there exists a sequence of holomorphic maps $ \; \widehat{\mathcal{X}}\xrightarrow {\quad \widehat{\psi}} \mathcal{X} \xrightarrow {\quad \varphi} \mathbb{P}^{1} \; $
with $ \; \varphi\circ \widehat{\psi}=\widehat{\varphi}$. Let $ \; {\mathbb C}({\mathcal X}) \; $  be the field of meromorphic functions on $ \; {\mathcal X}. \; $  The Galois group of the cover $ \; \varphi \; $ is the Galois group associated to the Galois closure of the field extension $ \; {\mathbb C}({\mathcal X})/{\mathbb C}({\mathbb P}^1).\; $
An elementary property of the Galois group of the cover $ \; \varphi: \mathcal{X}\rightarrow \mathbb{P}^{1}$  is that it has a natural representation  as a transitive subgroup of the symmetric group $\textbf{S}_{p}$.
\vspace{2mm}\\
Now consider a Galois cover of degree $q$
$$
 \psi:\mathcal{Y}\rightarrow \mathcal{X}  .
$$
We denote by $ \; \widetilde{\varphi} \; : \mathcal{Z} \to \mathbb{P}^{1}$  the Galois closure of the composite cover $\varphi\circ\psi \; $ and  by $\mathcal{G}$ the corresponding Galois group  of $\widetilde{\varphi}$. Hence $\mathcal{G}$  has a natural representation  as a transitive subgroup of the symmetric group $\textbf{S}_{pq}$.
\vspace{2mm}\\
Some well known properties of the Galois closure are collected next, adapted to our situation.  From now on, $ \;  \mathcal{Z}_{\mathcal{L}} \; $ denotes the quotient of $ \;  \mathcal{Z} \; $ by the subgroup $ \; \mathcal{L} \leq \mathcal{G}$.

\begin{prop} \label{prop:Gclosure}
Let $\mathcal{Z}$ be the Riemann surface associated to $\widetilde{\varphi}. \; $ Then there are subgroups $\mathcal{N}$ and  $\mathcal{H}$ of $\mathcal{G}$ satisfying the following properties:
 \begin{enumerate}
 \item  $\mathcal{Z}_{\mathcal{N}}\cong \mathcal{Y},\:\: \mathcal{Z}_{\mathcal{H}}\cong \mathcal{X} \; $ and  $ \; \mathcal{Z}_{\mathcal{G}}\cong\mathbb{P}^{1}$.
 \medskip
 \item  $\mathcal{N}_{\mathcal{G}}=\{1\}$,   where $ \; \mathcal{N}_{\mathcal{G}} = \core_{\mathcal{G}}(\mathcal{N}). \; $ In particular, if $ \; \mathcal{N} \unlhd G, \; $ then $ \; \mathcal{N} = \{ 1 \} \; $ and
 $\mathcal{Z}\cong \mathcal{Y}$.
 \medskip
  \item  $ \; \vert \mathcal{G} : \mathcal{H} \vert = p. \; $
 \medskip
 \item If $ \; \mathcal{K} := \mathcal{H}_{\mathcal{G}}, \; $ then $\mathcal{G}/\mathcal{K}\; $ is isomorphic to a transitive subgroup of $ \; \textbf{S}_p$. Also, $\mathcal{Z}_{\mathcal{K}} \to \mathcal{Z}_{\mathcal{G}}= \mathbb{P}^1$ is the Galois closure of the cover $\mathcal{X} \to \mathbb{P}^1$, with Galois group $\mathcal{G}/\mathcal{K}$ and $\mathcal{X} = \mathcal{Z}_{\mathcal{K}}/(\mathcal{H}/\mathcal{K})$.
  \medskip
\item $ \; \mathcal{N} \unlhd \mathcal{H} \; \; $ and $ \; \; \vert \mathcal{H} : \mathcal{N} \vert = q.$
 \end{enumerate}
\end{prop}

\begin{rem} \label{monodromia}
	The Monodromy representation of the cover $\varphi\circ\psi: \mathcal{Y}\rightarrow \mathcal{X} \rightarrow \mathbb{P}^{1} \;  $ is the natural group homomorphism $ \; \rho: \Pi_{1}(\mathbb{P}^1\setminus B, x_0) \rightarrow  \textbf{S}_{pq}\; $ with transitive image (isomorphic to $\mathcal{G}$) in  $ \;  \textbf{S}_{pq}.\; $
	It is well known that this representation is equivalent to the permutational representation given by the action of $\mathcal{G}$ on the right cosets of $ \; \mathcal{N} \; $ in $\mathcal{G}. \; $
	\vspace{2mm}\\
	We
	will use the same letters  $\mathcal{G}$, $\mathcal{H}$  and  $\mathcal{N}$  to denote  their corresponding images in $\textbf{S}_{pq}$.
	
	\bigskip
	
	The following diagram illustrates the relationship between covers and subgroups,
	{\tiny{
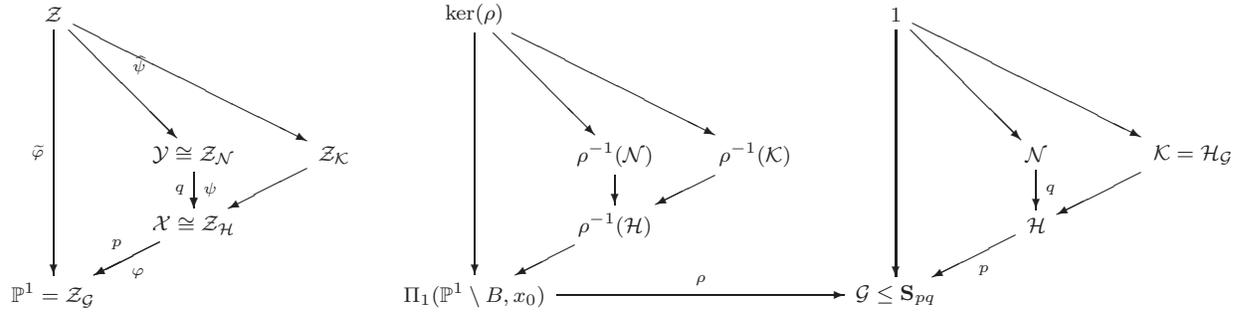
			$$
			\dgHORIZPAD=1em
			\dgVERTPAD=1.8ex
			\divide\dgARROWLENGTH by4
			\begin{diagram}
				\node{\mathcal{Z}}\arrow{sse,t}{\widehat{\psi}}\arrow[4]{s,l}{\widetilde{\varphi}}\arrow[2]{se}\node{}\node{}\node{\ker(\rho)}\arrow[4]{s}\arrow[2]{se}\arrow{sse}\node{}\node{}\node{1}\arrow[4]{s}\arrow{sse}\arrow[2]{se}\\
				\node{}\node{}\\
				\node{}\node{\mathcal{Y} \cong \mathcal{Z}_{\mathcal{N}}}\arrow{s,b}{\psi}\arrow{s,t}{q}\node{\mathcal{Z}_{\mathcal{K}}}\arrow{sw}\node{}\node{\rho^{-1}(\mathcal{N})}\arrow{s}\node{\rho^{-1}(\mathcal{K})}\arrow{sw}\node{}\node{\mathcal{N}}\; \; \; \; \; \arrow{s,r}{q}\node{\:\:\:\; \; \; \mathcal{K}={\mathcal H}_{\mathcal{G}}}\arrow{sw,r}{}\\
				\node{}\node{\mathcal{X}\cong \mathcal{Z}_{{\mathcal H}}}\arrow{sw,b}{\varphi}\arrow{sw,t}{p}\node{}\node{}\node{\rho^{-1}({\mathcal H})}\arrow{sw}\node{}\node{}\node{{\mathcal H}}\arrow{sw,b}{p}\\
				\node{\mathbb{P}^{1}=\mathcal{Z}_{\mathcal{G}}}\node{}\node{}\node{\Pi_{1}(\mathbb{P}^1\setminus B, x_0)}\arrow[3]{e,t}{\rho}\node{}\node{}\node{\mathcal{G}\leq \textbf{S}_{pq}}
			\end{diagram}$$}}
\end{rem}

\begin{prop}   \label{prop:kgi}
Let $ \; \{g_1, g_2, \dots , g_r \} \; $ be a set of geometric generators of  $\mathcal{G}$ given by  its action on $\mathcal{Z} \; $ and $ \; G_i = \langle g_i \rangle. \; $
 If the cover $ \; \psi:\mathcal{Y}\rightarrow \mathcal{X} \; $ is unramified, then $ \; \mathcal{K}\cap G_i= \{ 1 \}, \; $  for all  $ \;
 \; \; 1 \leq i \leq r  \; $.
\end{prop}

\begin{proof}
See Proposition 3.3 of \cite{Carocca;Romero}
\end{proof}	

\section{On the structure of the Galois group of the Galois closure of composite covers}

We recall the notation:   $ \; \; \widetilde{\varphi} \; $ is the Galois cover of the composite cover $\varphi\circ\psi. \; $
The group $ \; \mathcal{G} \; $ is the corresponding Galois group of $ \; \widetilde{\varphi}. \; $ Also $ \; \mathcal{Z} \; $ is the compact Riemann surface associated to $\widetilde{\varphi} \; $ and the subgroups $ \; \mathcal{N} \; $ and  $ \; \mathcal{H} \; $ of $\mathcal{G}$ correspond to  $ \; \mathcal{Z}_{\mathcal{N}}\cong \mathcal{Y} \; \; $ and $ \; \;  \mathcal{Z}_{\mathcal{H}}\cong \mathcal{X}$ respectively, with $\mathcal{K} := \mathcal{H}_{\mathcal{G}}$.
We are using the same letters  $ \; \mathcal{G}$, $ \; \mathcal{H} \; $,  $ \; \mathcal{N} \; $ and $\mathcal{K}$  to denote  their images in $\textbf{S}_{pq} \,$.
\vspace{2mm}\\
From now on we assume that
$ \psi $ is  an  unramified Galois cover of degree $ \; q \; $  and ${ \varphi}$ is a totally ramified cover of degree $\; p, \; $ for  different prime numbers $q$ and $p$ with $p \geq 3$.

\begin{prop}  \label{prop:alt}
The following  properties hold.
\begin{enumerate}
\item   The action of $\mathcal{G}$ on $ \; \mathcal{Z} \; $ induces a geometric presentation of $\; \mathcal{G} \; $ given by
$$\mathcal{G}=\left\langle g_{1}, g_{2},\cdots, g_{r} \; \; / \; \;  \prod_{i=1}^{r}g_{i}=1, \; \; g_{i}^{p}=1, \;   \; \;  i = 1, 2,...,r,  \;\; \; \mathcal{R} \right\rangle$$
where $\mathcal{R}$ is a set of  appropriate  relations on $ \; \{ g_1, g_2, ..., g_r \}.$

\item  The corresponding action of $ \; \mathcal{G}/\mathcal{K} \; $ on $ \; \mathcal{Z}_{\mathcal{K}} \; $ induces a geometric
presentation of $\; \mathcal{G}/\mathcal{K} \; $ given by
$$\mathcal{G}/\mathcal{K}=\left\langle g_{1}\mathcal{K}, g_{2}\mathcal{K} , .... , g_{r}\mathcal{K} \; \; / \; \;  \prod_{i=1}^{r}g_{i} = 1,  \; \; g_{i}^{p}=1, \;   \; \;  i = 1, 2,...,r, \;\; \; \mathcal{R}^{\prime}  \right\rangle
$$
where $\mathcal{R}^{\prime} \; $ is a set of  appropriate  relations on the set of cosets $ \; \{ g_1\mathcal{K}, g_2\mathcal{K}, ..., g_r\mathcal{K} \}.$

\item The corresponding image in $ \; \textbf{S}_{pq} \; $ for each $ \; g_i \; $ has a cycle structure given as a product of $ \; q \; $ disjoint cycles of length $ \; p. \; $ Hence $ \;   \mathcal{G} \leq {\bf A}_{pq}$.

\item The corresponding image in $ \; \textbf{S}_{p} \; $ for each $ \; g_i\mathcal{K} \; $ is a cycle of length $ \; p. \; $ Hence $ \;   \mathcal{G}/\mathcal{K} \leq {\bf A}_{p}.$
\end{enumerate}
\end{prop}

\begin{proof}
Since $\varphi$ is  a totally ramified cover of prime degree $p$, the local monodromy at each of its branch points is a cycle of
length $p$; moreover, since $\psi$ is an unramified Galois cover of degree $q$,   the local monodromy $g_i$ of $\varphi \circ \psi$ at each branch point $P_i$ is the product  of $ \; q \; $ disjoint cycles of length $p$.

Since these $g_i$  (the products of  $  q  $ disjoint cycles of length $p$) generate $\mathcal{G} \leq \textbf{S}_{pq}$,  the action of $\mathcal{G}$ on $\mathcal{Z}$ has signature $(0; p , \ldots, p)$, and therefore, according to Theorem \ref{thm:geomgens}, it has a presentation as in (1), and furthermore   $ \;   \mathcal{G} \leq {\bf A}_{pq}$, thus proving (1) and (3).

Since we know from Proposition \ref{prop:kgi} that $\mathcal{K} \cap \langle g_i \rangle =  \{1\}$ for all $i$, and from Proposition \ref{prop:Gclosure} that $\mathcal{Z}_{\mathcal{K}} \to \mathcal{X} \to \mathcal{Z}_{\mathcal{G}}= \mathbb{P}^1$ is the Galois closure of the totally ramified cover of degree $p$ $\mathcal{X} \to \mathbb{P}^1$,  (2) and (4) follow immediately.
\end{proof}

\medskip

We now start the study of the algebraic properties of the groups involved that will lead to their classification.

\begin{prop}\label{core}
Recall that  $ \; \mathcal{K} :=  \mathcal{H}_{\mathcal{G}}.\; $ Then

\begin{enumerate}
\item  $ \; \mathcal{K} \; $ is an elementary abelian  $ \; q$-group.
\item If $ \; {\mathcal P} \; $ is a Sylow  $p$-subgroup of $ \; \mathcal{G}, \; $ then $ \; \vert  {\mathcal P}  \vert = p \; $ and $ \; \mathcal{G} = \mathcal{H} {\mathcal P} $.  Furthermore, $ {\mathcal P} $ is not normal in $\mathcal{G}$.

\item   $\mathcal{G}/\mathcal{K}$ is a simple transitive group of degree $p$.

\item   $ \; {\mathcal N} \leq {\mathcal K} \; $ if and only if  $\mathcal{H}$ is normal in $\mathcal{G}$.
\end{enumerate}
\end{prop}

\begin{proof}
 Let $\{ \; x_{1}, x_{2},\cdots, x_{p} \; \} $  be a right transversal of $\mathcal{H}$ in $\mathcal{G}. \; $ Considering the group monomorphism $$ \; \Phi : \mathcal{K} \rightarrow (\mathcal{H}/\mathcal{N})^{x_1} \times (\mathcal{H}/\mathcal{N})^{x_2} \times \ldots \times (\mathcal{H}/\mathcal{N})^{x_p} \; $$ defined by
$ \; \Phi(k) = (\mathcal{N}^{x_1}k, \mathcal{N}^{x_2}k, \ldots , \mathcal{N}^{x_p}k), \; $ we have that $ \ker(\Phi) = \mathcal{N}_\mathcal{G} = \{1\} \; \; $ and hence $ \; \mathcal{K} \cong \mathcal{K}/\ker(\Phi) \leq ({\mathbb Z}_q)^p.$

\medskip

Since  $ \; \mathcal{G}/\mathcal{K} \leq \textbf{A}_{p} \; \; $ and  $ \; \mathcal{K} \; $  is a  $ \; q$-group, it follows that $ \; \vert  {\mathcal P}  \vert = p$; but then $ \; \vert \mathcal{G} : \mathcal{H} \vert = p$ implies that $ \; \mathcal{G} = \mathcal{H} {\mathcal P} $.  We know from Proposition \ref{prop:alt} that $\mathcal{G}$ is generated by elements of order $p$, and hence $ {\mathcal P} $ cannot be normal in $\mathcal{G}$.

\medskip

To show $\mathcal{G}/\mathcal{K}$ is simple, suppose that  $ \; \{1\} \ne R/\mathcal{K} \unlhd  \mathcal{G}/\mathcal{K}$.  As $ \; \mathcal{K} =  \mathcal{H}_{\mathcal{G}}\; $ and $ \; \vert \mathcal{G} : \mathcal{H} \vert = p \; $, then
$ \; R \nleq \mathcal{H} \; $ and  $ \;  \mathcal{G} = \mathcal{H}R. \; $ Since $ \; R \unlhd  \mathcal{G}  \; $ and $  \;  {\mathcal P}  \leq R, \; $ it follows that every conjugate  of $ \;  {\mathcal P}  \; $ in $ \;  \mathcal{G} \; $ is contained in  $ \; R.$
But $ \;  \mathcal{G} \; $ is generated by elements of order $ \; p \; $, and we conclude that $ \; R =   \mathcal{G}. \; $ Together with Proposition \ref{prop:Gclosure}, we have obtained that  $ \;  \mathcal{G}/\mathcal{K} \; $ is a simple transitive group of degree $p$.

\medskip

Finally, suppose that $ \; {\mathcal N} \leq {\mathcal K}$. Since $ \; \vert {\mathcal H} : {\mathcal N} \vert = q$, then either $ \; {\mathcal K} = {\mathcal H} \; $ or $ \; {\mathcal K} = {\mathcal N}$. Assuming $ \; {\mathcal K} = {\mathcal N}, \; $ we obtain $ \; {\mathcal N} = {\mathcal K} = {\mathcal N}_{\mathcal{G}} = \{1\}$ and hence $ \;  \vert {\mathcal{G}} \vert = pq$. Since we know that $ \; P \not \!\! \unlhd \:  {\mathcal{G}} \; $, it is a standard fact that then $ \; {\mathcal H} \unlhd {\mathcal{G}}$,  a contradiction to $\mathcal{K} =1$, and the conclusion follows since the other implication is trivial.
\end{proof}

The next result follows immediately from the previous Proposition; we prefer to write it down explicitly as it provides the basis for the construction that will allow us to describe the monodromy action of our factorized cover.

\begin{cor}  \label{cor:ntausigma}
Let $ {\mathcal P}  = \langle \sigma \rangle$ denote a Sylow $p$-subgroup of $\mathcal{G}$. Then there exists $\tau \in \mathcal{H}$ such that $\tau^q \in \mathcal{N}$, $\mathcal{H} =  \mathcal{N}\, \langle \tau	\rangle$ and $\mathcal{G} = \mathcal{H} \, \langle \sigma \rangle= \mathcal{N}\, \langle \tau	\rangle\, \langle \sigma \rangle$.

If $\mathcal{H}$ is not normal in $\mathcal{G}$ and if $\mathcal{K} \neq 1$, then $\mathcal{H} = \mathcal{K}\, \mathcal{N}$ and $\mathcal{G} = \mathcal{K}\, \mathcal{N}  \, \langle \sigma \rangle$.
\end{cor}

\medskip

\begin{rem} \label{rem17}
	We will now use the previous Corollary and Proposition to give a description of the right cosets of $\mathcal{N}$ in $\mathcal{G}$ that will allow us to understand the monodromy action of our factorized cover.
	
	By Corollary \ref{cor:ntausigma},  there exists $\tau \in \mathcal{H}$ such that $\tau^q \in \mathcal{N}$, $ \; \mathcal{H} = \mathcal{N} \, \langle \tau \rangle \; $ and $ \;  \mathcal{G} =  \mathcal{N}\, \langle \tau \rangle \, \langle \sigma \rangle$, with $\sigma$ an element of order $p$ in $\mathcal{G}$.

	Consider  $ \; \{ 1 = \sigma^0, \sigma, \dots , \sigma^{p-1} \; \}$   as a right transversal of $\mathcal{H}$ in
	$\mathcal{G} \; $ and  $ \{\; 1 = \tau^0, \tau, \dots , \tau^{q-1} \}$  as a right transversal of $\mathcal{N}$ in $\mathcal{H} $. Then the set
	$ \; \; \{\tau^{a-1} \sigma^{b-1} \; \; \; \; / \; \;  \; a = 1,...,q,\;  b = 1,..., p \}\; \; $
	is a right transversal  of  $\mathcal{N}$ in $\mathcal{G}. \; $
\end{rem}

Let $\Omega = \Delta_{1}\cup \Delta_{2}\cup\cdots \cup\Delta_{p}, \; $ where $ \; \Delta_b = \{\mathcal{N}\tau^0\sigma^{b-1}, \mathcal{N}\tau \sigma^{b-1}, \ldots , \mathcal{N} \tau^{q-1} \sigma^{b-1} \} \; $ for $ \; b = 1, \ldots , p$.

\begin{equation} \label{eq:pict}
	\begin{array}{lllll}
		\mathcal{N} & \mathcal{N}\sigma & \mathcal{N}\sigma^2 & \dots & \mathcal{N} \sigma^{p-1} \\
		\mathcal{N}\tau  & \mathcal{N}\tau \sigma & \mathcal{N}\tau \sigma^2 & \dots & \mathcal{N}\tau \sigma^{p-1} \\
		\mathcal{N}\tau^2 & \mathcal{N}\tau^2 \sigma & \mathcal{N}\tau^2 \sigma^2 & \dots & \mathcal{N}\tau^2 \sigma^{p-1} \\
		\vdots \vdots \vdots & \vdots \vdots \vdots & \vdots \vdots \vdots & \dots & \vdots \vdots \vdots\\
		\mathcal{N}\tau^{q-1} & \mathcal{N}\tau^{q-1} \sigma & \mathcal{N}\tau^{q-1} \sigma^2 & \dots & \mathcal{N}\tau^{q-1} \sigma^{p-1} \\
		\\
		\Delta_1 & \Delta_2 & \Delta_3 & \dots & \Delta_p
	\end{array}
\end{equation}

\medskip

With these choices we can describe the (right) action of several subgroups of $ \;  \mathcal{G} \; $ on the sets $\Delta_b$ as follows.

\begin{prop}  \label{prop:action}
	For each  $\; 1 \leq b \leq p  \; $, the subgroups $ \; \mathcal{N}_b =\sigma^{-(b-1)}\mathcal{N}\sigma^{(b-1)} \; $,  $ \; \mathcal{H}_b = \sigma^{-(b-1)}\mathcal{H}\sigma^{(b-1)} $ and $\mathcal{K}$ act on the set $\Delta_{b}$ (under the right action).
	
	Furthermore, the action of $\mathcal{N}_b$ on $\Delta_b$ is trivial.
	
	Also, $\sigma$ acts on  $\Omega$ by $\Delta_b \, \sigma = \Delta_{b+1}$ for $1 \leq b \leq p-1$ and $\Delta_p \, \sigma = \Delta_{1}$.
\end{prop}

\begin{proof}
	Let $x \in \mathcal{N}_b$ and $\delta \in \Delta_b$. Then $x = 	\sigma^{-(b-1)} \, n  \, \sigma^{(b-1)}$ for some $n \in \mathcal{N}$, $\delta = \mathcal{N} \, \tau^j \, \sigma^{b-1}$ for some $0 \leq j \leq q-1$, and therefore
	\begin{equation*}
		\delta \, x = (\mathcal{N} \, \tau^j \, \sigma^{b-1})(\sigma^{-(b-1)} \, n  \, \sigma^{(b-1)}) = \mathcal{N} \,  \tau^j \, n \, \sigma^{(b-1)} =  \mathcal{N} \, n_1\, \tau^j \, \sigma^{(b-1)}	= \mathcal{N} \, \tau^j \, \sigma^{b-1} = \delta
	\end{equation*}	
	where $\tau^j n = n_1\tau^j $ for some $n_1 \in \mathcal{N}$, since $\mathcal{N}$ is normal in $\mathcal{H}$.	
	
	Now consider $y \in \mathcal{H}_b$. Then $y = 	\sigma^{-(b-1)} \, h  \, \sigma^{(b-1)}$ for some $h \in \mathcal{H}$. Since $\mathcal{H} = \mathcal{N}\, \langle \tau \rangle$, by Remark \ref{rem17},  and $\mathcal{N}$ is normal in $\mathcal{H}$, it follows that we can write $h = n \, \tau^k = \tau^k \, n_1$ for some $n, n_1 \in \mathcal{N}$ and for some $0 \leq k \leq q-1$, and hence $y = 	\sigma^{-(b-1)} \, n \, \tau^k  \, \sigma^{(b-1)} = \sigma^{-(b-1)} \, \tau^k  \, n_1 \,  \sigma^{(b-1)} $. Therefore
	\begin{align*}
		\delta \, y   & = (\mathcal{N} \, \tau^j \, \sigma^{b-1})(\sigma^{-(b-1)} \, \tau^k \, n_1 \, \sigma^{(b-1)})   \\
		& = \mathcal{N} \, \tau^{j+k} \,  n_1  \, \sigma^{(b-1)} = \mathcal{N} \, \tau^{j+k}  \, \sigma^{(b-1)} \in \Delta_b \, ,
	\end{align*}
	where the last equality follows as before, since $\mathcal{N}$ is normal in $\mathcal{H} = \mathcal{N} \, \langle \tau \rangle$.
	
	Since $\mathcal{K} = \displaystyle\bigcap_{b=1}^p \mathcal{H}_b$, it follows that $\mathcal{K}$ acts on each $\Delta_b$.

	The last assertion is inmediate, as
	$$
	\delta \, \sigma =  ( \mathcal{N} \, \tau^j \, \sigma^{b-1}) \, \sigma =  \mathcal{N} \, \tau^j \, \sigma^{b} \in
	\begin{cases}  \Delta_{b+1} & \textup{if } 1 \leq b \leq p-1; \\
		\Delta_1 & \textup{if }  b=p.
	\end{cases}
	$$
\end{proof}

\medskip

We will now interpret the actions described above in accordance to Remark \ref{monodromia}. That is, we will describe the monodromy representation of $\mathcal{G} \leq \textbf{S}_{pq}$.

For this purpose we  identify the set $ \; \Delta_b \; $ with the set $ \; \{ b, p + b, 2p + b, \dots , (q-1)p + b \} = \Delta_b$, $\Omega$ with the set $\{1, 2, \ldots , pq\}$, and diagram \eqref{eq:pict} with the following one.

$$\begin{array}{ccccc}
	1 & 2 & 3 & \dots & p \\
	p+ 1 & p + 2 & p + 3 & \dots & 2p \\
	2p + 1 & 2p + 2 & 2p + 3 & \dots & 3p \\
	\vdots \vdots & \vdots \vdots & \vdots \vdots & \dots & \vdots \vdots\\
	(q-1)p + 1 & (q-1)p +2 & (q-1)p + 3 & \dots & pq \\
	\\
	\Delta_1 & \Delta_2 & \Delta_3 & \dots & \Delta_p
	\vspace{3mm}\\
\end{array} $$

\medskip

We now introduce an auxiliary subgroup of $\textbf{S}_{pq}$ that will allow us to give a complete characterization of $\mathcal{G}$.

Consider the elements $ \; \epsilon_b \in \textbf{S}_{pq} \; $ and $\epsilon \in \textbf{S}_{pq}$ given by
$$
\epsilon_b = (b,p+b,2p+b,\dots, (q-1)p+b) \; \; \; \mbox{with} \; \; 1 \leq b \leq p,
$$
$$ \epsilon = \epsilon_1 \epsilon_2 \cdots \epsilon_p \, , $$
and the abelian $ \; q$-group $$ \textbf{R} = \langle \epsilon_1, \epsilon_2, \dots, \epsilon_p \rangle \cong \mathbb{Z}_q^p\leq \textbf{S}_{pq}\, .$$

\medskip

\begin{prop} \label{prop:norm}
	We have that $ \; {\mathcal K}  \leq {\textbf R} \; $ and $ \;  \mathcal{G} \leq \textbf{C}_{\textbf{S}_{pq}}({\langle \epsilon \rangle })\; $ (the centralizer of $ \; {\langle \epsilon \rangle }  \; $ in the symmetric group $ \; {\textbf S}_{pq}$).
\end{prop}

\begin{proof}
	According to  Proposition \ref{prop:action},  $ \; {\mathcal H}_b  \; $ acts on $ \; \Delta_b \; $ (hence also on $\Omega \setminus \Delta_b$). Furthermore, from the proof of Proposition \ref{prop:action},  the action of   $y_b \in {\mathcal H}_b \; $  on $ \; \Omega \; $ is given by $$y_b = (\epsilon_b)^r\eta_b $$
	where $ \; \eta_b$ represents the action of $ \; y_b \; $ on $ \; \Omega \setminus \Delta_b. \; $
	\vspace{2mm}\\
	In particular, 	 $ \; {\mathcal H}_b \leq \textbf{C}_{\textbf{S}_{pq}}({\langle \epsilon_b \rangle })\;  \; $ and $ \; {\mathcal K} \leq \displaystyle\bigcap_{b = 1}^{p} \textbf{C}_{\textbf{S}_{pq}}({\langle \epsilon_b \rangle }) = \textbf{R}.$
	\vspace{2mm}\\
	It also follows from  the proof of Proposition \ref{prop:action} that
	 the action of $ \; \sigma \; $ on $ \Omega \; $ is represented  by
	$$\sigma = (1,2,\dots , p)(p+1, p+2,\dots , 2p)\cdots((q-1)p + 1, (q-1)p+2,\dots, pq) $$
	and hence
	$$
	\sigma^{-1} \, \epsilon_b \, \sigma = \epsilon_{b+1} \; \mbox{for } \; 1 \leq b \leq p-1 \; \; \mbox{and} \; \; \sigma^{-1} \, \epsilon_p \, \sigma = \epsilon_1 \, ;$$
	that is, $\sigma \in    \textbf{C}_{\textbf{S}_{pq}}({\langle \epsilon \rangle })$.
	\vspace{2mm}\\
	Since $ \; \mathcal{G} = {\mathcal H}_b \langle \sigma \rangle $ for each $b \in \{1, \ldots , p\}$, for every $ \; g \in \mathcal{G}$ there exist $ \; y_b \in {\mathcal H}_b \; $ and $ \; 1 \leq i_b \leq p \; $ such that  $ \; g = y_b\sigma^{i_b} $.  Hence
	$$g^{-1} \,\epsilon_b \, g = (y_b\sigma^{i_b})^{-1} \, \epsilon_b \, y_b \, \sigma^{i_b} = \sigma^{-i_b}\, \epsilon_b \, \sigma^{i_b}.$$
	In this way
	\begin{align*}
	g^{-1} \,\epsilon \, g & = g^{-1} \, ( \epsilon_1 \epsilon_2 \cdots \epsilon_p) \, g =(g^{-1} \,  \epsilon_1  \, g)(g^{-1} \,  \epsilon_2  \, g) \cdots (g^{-1} \,  \epsilon_p  \, g) =\\
      	& = (\sigma^{-i_1}\, \epsilon_1 \, \sigma^{i_1})  (\sigma^{-i_2}\, \epsilon_2 \, \sigma^{i_2}) \cdots  (\sigma^{-i_p}\, \epsilon_p \, \sigma^{i_p}) = \epsilon \, ,
	\end{align*}
	where the last equality holds since each $\sigma^{-i_b}\, \epsilon_b \, \sigma^{i_b} = \epsilon_k$ for some $k$ and the cycle structure of $\epsilon$ and the cycle structure  of $g^{-1} \,\epsilon \, g$ coincide.
\end{proof}

\medskip

It is well known that
$$
\textbf{C}_{\textbf{S}_{pq}}(\langle \epsilon \rangle ) = \textbf{R} \rtimes \textbf{S}_p \subset \textbf{S}_{pq}  \,
$$
where
$$
\textbf{S}_p = \langle   (1,2, \ldots ,  p)(p+1, \ldots , 2p) \ldots ((q-1)p+1, \ldots, qp), (1,2)(p+1,p+2) \ldots ((q-1)p+1, (q-1)p+2)  \rangle
$$
acts by conjugation on $\textbf{R}$. This action is not irreducible, as it decomposes into the trivial action of $\textbf{S}_p$ on $\langle \epsilon \rangle$ and the standard one on
$$\textbf{J} = \langle \epsilon_1(\epsilon_2)^{-1}, \epsilon_2(\epsilon_3)^{-1}, \dots , \epsilon_{p-1}(\epsilon_p)^{-1} \rangle \leq {\bf A}_{pq}.
$$

\medskip

We have just shown that $\mathcal{G} \leq \textbf{C}_{\textbf{S}_{pq}}(\langle \epsilon \rangle )$, but we also know, from Proposition \ref{prop:alt}, that $ \;   \mathcal{G} \leq \textbf{A}_{pq}$. Our next result puts these two pieces of information together.

\medskip

\begin{prop} \label{sobreK}
Let
$$
 \textbf{T} :=  \textbf{J} \rtimes \textbf{A}_p \leq  \textbf{R} \rtimes \textbf{S}_p,
$$	
where $\textbf{J} = \langle \epsilon_1(\epsilon_2)^{-1}, \epsilon_2(\epsilon_3)^{-1}, \dots , \epsilon_{p-1}(\epsilon_p)^{-1} \rangle \leq \textbf{R}$ and the action of $\textbf{A}_p$ on $\textbf{J}$ is the restriction of the above action of $\textbf{S}_p$ on $\textbf{R}$.

Then 	\begin{enumerate}
	\item $\textbf{T}  \unlhd  \textbf{C}_{\textbf{S}_{pq}}(\langle \epsilon \rangle )	\cap  \textbf{A}_{pq}$;

\medskip

	\item   $ \;  {\mathcal{G}} \leq  {\textbf T}$;

\medskip

	\item $ \; {\mathcal K} \leq \textbf{J}. \; \; $ In particular, $ \; \vert {\mathcal K} \vert \leq q^{p-1}.$

\medskip

	\item $ \; {\mathcal K} = \mathcal{G}  \cap {\textbf J}$.
\end{enumerate}
\end{prop}

\begin{proof}
	
	The first statement is clear.
	
	\medskip
	
	Since $ \; \vert \textbf{C}_{\textbf{S}_{pq}}(\langle \epsilon \rangle ) : {\textbf T} \vert = 2q$ and since  $ \; {\textbf T} \unlhd \textbf{C}_{\textbf{S}_{pq}}(\langle \epsilon \rangle )$, $ \; {\textbf T} \; $ contains every elements of order $ \; p \; $ of $ \; \textbf{C}_{\textbf{S}_{pq}}(\langle \epsilon \rangle )$.
	As $ \;  \mathcal{G} \; $ is generated by elements of order $ \; p, \; $ we conclude that $ \;  \mathcal{G} \leq \textbf{T}. $

    \medskip

	Since we already know from Proposition \ref{prop:norm} that  $ \; {\mathcal K} \leq \textbf{R}$, if we assume $ \; {\mathcal K} \not \leq {\textbf J}, \; $ then we obtain $ \; {\textbf R} = {\mathcal K}{\textbf J} \leq {\textbf T}, \; $ a contradiction.

    \medskip

	Finally, since  $ \;  \mathcal{G} \leq {\textbf T} $ and $  {\mathcal K} \leq \textbf{J} \unlhd {\textbf T}$, it follows that   $ \; {\mathcal K} \leq  \textbf{J} \cap \mathcal{G} \unlhd \mathcal{G}$.
	As $ \; {\mathcal{G}}/\mathcal{K} \; $ is a simple  group and $ \; \mathcal{G} \neq  \textbf{J} \cap \mathcal{G}$,  the result holds.
\end{proof}

We can now give the general characterization of the monodromy group $\mathcal{G}$.

\begin{thm}
	 $ \;  \mathcal{G}  \cong {\mathcal K} \rtimes \mathcal{U}\; $ with $ \; \mathcal{U} \; $ a simple  transitive group of
	degree $p.$
	
	\end{thm}

\begin{proof}
	Since $ \; \textbf{J} \unlhd \textbf{T} :=  \textbf{J} \rtimes \textbf{A}_p $ and since it follows from Proposition \ref{sobreK} (2) that  $ \mathcal{G} \leq \textbf{T} $, then $ \;  \textbf{J}\mathcal{G}  \leq \textbf{T}$.
	
	Applying Proposition \ref{sobreK} (4), we obtain
	
	$$ \textbf{J} \mathcal{G}/\textbf{J} \cong  \mathcal{G}/( \mathcal{G} \cap \textbf{J}) =  \mathcal{G}/\mathcal{K}\, .
	$$

	Hence	
	$$
	\textbf{J} \mathcal{G} = \textbf{J} \rtimes (\textbf{J} \mathcal{G} \cap {\bf A }_p) \; \; \; \;  \mbox{and}  \; \; \; \;  \textbf{J} \mathcal{G} \cap \textbf{A}_p \cong ( \textbf{J} \mathcal{G} \cap \textbf{A}_p)/ ( \textbf{J} \mathcal{G} \cap \textbf{A}_p \cap {\textbf J} )  \cong \textbf{J} \mathcal{G}/\textbf{J} \cong    \mathcal{G}/{\mathcal K} \, ,
	$$
	and therefore $\mathcal{U}:= \; \textbf{J} \mathcal{G} \cap \textbf{A}_p \, \cong \mathcal{G}/{\mathcal K}  $ is
	a simple  transitive  group of  degree $p$, as we know from Proposition \ref{core}.
	
	\medskip
	
	To complete the first part of the proof, we will now verify that $\mathcal{G} \cong {\mathcal K}  \rtimes \mathcal{U}$ by first showing that ${\mathcal K} \mathcal{U} = {\mathcal K} \rtimes \mathcal{U} \leq  \textbf{J} \mathcal{G}$ and then exhibiting an explicit isomorphism from $\mathcal{G}$ to ${\mathcal K} \mathcal{U}$.
	
	\medskip
	
	Since ${\mathcal K}$ is normal in $\mathcal{G}$ and in the abelian group $\textbf{J}$, it follows that $ {\mathcal K} \unlhd \textbf{J} \mathcal{G}$ and hence  $ {\mathcal K} \, \mathcal{U} = \; {\mathcal K} (\textbf{J} \mathcal{G} \cap \textbf{A}_p) \leq  \textbf{J} \mathcal{G}$. But then
	$$
	{\mathcal K} \mathcal{U}/{\mathcal K} = {\mathcal K}(\textbf{J} \mathcal{G} \cap \textbf{A}_p)/{\mathcal K} \cong \textbf{J} \mathcal{G} \cap \textbf{A}_p/ ( \textbf{J}\mathcal{G} \cap \textbf{A}_p \cap {\mathcal K}) \cong  \textbf{J}\mathcal{G} \cap \textbf{A}_p = \mathcal{U},
	$$
	since $\mathcal{U} \cap {\mathcal K} =\textbf{J}\mathcal{G} \cap \textbf{A}_p \cap {\mathcal K} =1$ as ${\mathcal K} \leq \textbf{J}$.  Therefore ${\mathcal K} \, \mathcal{U} = {\mathcal K} \rtimes \mathcal{U}$.
	
	\medskip
	
	To write down an isomorphism from $\mathcal{G}$ to ${\mathcal K} \mathcal{U}$, first decompose  $ \; \mathcal{G} \; $ as a disjoint union of right cosets of ${\mathcal K}$: $ \; \mathcal{G} = {\mathcal K}h_{1} \cup {\mathcal K}h_{2} \cup \cdots \cup {\mathcal K}h_r$, with $h_j \in \mathcal{G}$.
	Since $ \; \mathcal{G} \leq \textbf{J} \mathcal{G} =  \textbf{J} \rtimes \mathcal{U}$, each $h_j$  can be written as $ \; h_j = a_j\, b_j \; $ with unique $ \; a_j \in \textbf{J} \; \; $ and  $ \; \; b_j \in \mathcal{U}$.
	
	Hence each $ \; g \in \mathcal{G} \; $  may be written as  $$ \; g = k\, a_j \,b_j\; $$ with  $ \; \; k \in {\mathcal K}, \; \; a_j \in \textbf{J} \; \; $ and  $ \;  b_j \in \mathcal{U}$.
	
	Now consider  $ \; \phi : \mathcal{G} \to  {\mathcal K} \mathcal{U} $ defined by  $ \; \phi(g) = \phi(k\, a_j \, b_j)  = k\, b_j$.
	
	Let $\; g_1 = k_1a_jb_j$ and  $g_2 = k_2a_sb_s$ be in $\mathcal{G}$. Then, using that  $\mathcal{K}$ and $\textbf{J}$ are normal subgroups of $\textbf{J} \mathcal{G} = \textbf{J} \rtimes \mathcal{U}$, we see that	
	 $$ \; g_1g_2 = k_1a_jb_jk_2a_sb_s = k_1a_j(k_2)^{\prime}b_ja_sb_s = k_1(k_2)^{\prime}a_jb_ja_sb_s = k_1(k_2)^{\prime}a_j(a_s)^{\prime}b_jb_s
	$$
	and hence
	$$\phi(g_1g_2) = k_1(k_2)^{\prime}b_jb_s .
	$$
	
	But
	$$\phi(g_1)\phi(g_2) = k_1b_jk_2b_s = k_1(k_2)^{\prime}b_jb_s
	$$
	and we conclude that $ \; \phi \; $ is a group homomorphism.
	
	Observe that, if  $ \; \phi(g) = kb_j = 1 \; $ then $ \; k^{-1} = b_j \in {\mathcal K} \cap  \mathcal{U} =  \{ 1 \}$. Hence  $ \; \phi \; $ is injective.
	
	Finally, since $ \; \vert  {\mathcal K} \mathcal{U}  \vert = \vert  \mathcal{G}\vert$, we have obtained
	$$
	\mathcal{G}   \cong {\mathcal K} \rtimes \mathcal{U}.
	$$
\end{proof}

Recall that we have proven that $\mathcal{G}  \cong  \mathcal{K} \rtimes \mathcal{U}$, with $\mathcal{K} = \mathcal{H}_{\mathcal{G}}$ isomorphic to a subgroup of $\mathbb{Z}_q^{p-1}$ and $\mathcal{U}$ a transitive simple group of degree $p. \; $
From now on we  use the same letter $\mathcal{U}$ for the corresponding subgroup of $\mathcal{G}$ such that  $\mathcal{G}  = {\mathcal K}\: \mathcal{U}$ and $ \; {\mathcal K} \cap \mathcal{U} = \{ 1 \}$.

\section{The case $\mathcal{X} \to \mathbb{P}^1$ is a cyclic $p$-fold cover}

In order to prove part i) of the Main Theorem, we now consider the case when $\mathcal{X} \to \mathbb{P}^1$ is a cyclic $p$-fold cover.

It is clear that $\mathcal{Z}_{\mathcal{H}}\cong \mathcal{X} \xrightarrow {\quad \varphi}  \mathcal{Z}_{\mathcal{G}}\cong\mathbb{P}^{1}$
 is a cyclic $p$-fold cover if and only if $\mathcal{H}$ is a normal subgroup of $\mathcal{G}$, and this in turn is equivalent to $\mathcal{H} = \mathcal{K}$; to complete the proof of part i) we now prove the following result.

\begin{thm}
	$\mathcal{G}$ is solvable if and only if  $\mathcal{H}$ is normal in $\mathcal{G}$.
	\vspace{2mm}\\
	In this case $ \; {\mathcal U} \cong {\mathbb Z}_p \; $ and
	$$
	\mathcal{G} \cong \mathbb{Z}_q^s \rtimes \mathbb{Z}_p
	$$
	with $1 \leq s \leq p-1$.
	\vspace{2mm}\\
	Furthermore, $\varphi \circ \psi $ is a Galois cover if and only if  
	$$\mathcal{G} = \langle a, b : a^q = b^p =1 , b^{-1}ab = a^r\rangle \cong \mathbb{Z}_q \rtimes \mathbb{Z}_p
	$$
	with $p/(q-1)$ and $ 1 \ne r$ a primitive $p$-th root of unity in the field of $q$ elements.

\end{thm}

\begin{proof}
	If $\mathcal{G}$ is solvable, then  $ \;  \mathcal{G}/ \mathcal{K} \cong \mathcal{U} $ is a solvable and simple group;  hence, the order of  $ \;  \mathcal{G}/ \mathcal{K} \; $ is a prime number. Since $ \; p = \vert  \mathcal{G} : \mathcal{H} \vert \; $ divides $ \;  \vert  \mathcal{G}: \mathcal{K} \vert \; $, it follows that $ \;  \mathcal{H} = \mathcal{K}$, so $\mathcal{H}$ is normal in $\mathcal{G}$, and that $\mathcal{U} \cong \mathbb{Z}_p$. Also note that since $\mathcal{H} = {\mathcal K} \; $ is non trivial, it is isomorphic to $\mathbb{Z}_q^s$ for some $1 \leq s \leq p-1$.
	
	\medskip

	If $ \;  \mathcal{H} \unlhd \mathcal{G},  \; $ then $ \;  \mathcal{H} =  \mathcal{K} \; \; $ is an abelian $ \; q$-group and $ \;  \mathcal{G}/\mathcal{H} \; $ has order $ \; p$; in particular, it is an abelian group. Hence  $ \; \mathcal{G} \; $ is a solvable group. But we know from Proposition \ref{core} that $ \; \mathcal{G} = \mathcal{H} {\mathcal P} , \; $ with $ {\mathcal P} $ a $p$-Sylow subgroup of $ \mathcal{G}$, and hence
	$$
	\mathcal{G} \cong \mathbb{Z}_q^s \rtimes \mathbb{Z}_p
	$$
	with $1 \leq s \leq p-1$.
	
	\medskip
	
	To finish the proof, observe that $\varphi \circ \psi $ is a Galois cover if and only if $ \; {\mathcal N} = \{ 1\} \; $  if and only if $ \; \vert {\mathfrak G} \vert = pq \; $, since $ \;   \vert {\mathfrak G} \vert =  \vert {\mathfrak G} : {\mathcal H} \vert  \vert {\mathcal H} : {\mathcal N} \vert \vert {\mathcal N} \vert = p \cdot q \cdot \vert {\mathcal N} \vert$.
	
	Since we know from Proposition \ref{core} that   a $p$-Sylow subgroup of $ \mathcal{G}$ is not normal in $\mathcal{G}$, $\mathcal{G}$ must be the unique non abelian group of order $pq$, given by
			$$ G = \langle a, b : a^q = b^p =1 , b^{-1}ab = a^r\rangle \cong \mathbb{Z}_q \rtimes \mathbb{Z}_p
		$$
	as claimed.	
\end{proof}

\section{The non-Solvable case} \label{S:ns}

We have already shown that $\mathcal{G}$ is solvable if and only if $\mathcal{H}$ is normal in $\mathcal{G}$ if and only if $\mathcal{X} \xrightarrow {\quad \varphi} \mathbb{P}^{1}$ is a cyclic $p$-fold cover.

\medskip

In this section we assume that $\mathcal{G}$ is not solvable, or, equivalently, that $\mathcal{K} \subsetneq \mathcal{H}$, and prove part ii) of the main Theorem.
\medskip

A key ingredient is the following result.

\begin{thm} (Guralnick \cite{GU})\label{guralnick}
Let $ \; G \; $  be a simple non-abelian   transitive  group of
prime degree $ \; p \; $ and $H$ a subgroup of $G$ of index $p$.  Then $ \; G \; $  is one of the following groups:
\begin{enumerate}
\item  $ \; G  = {\mathbf A}_p \; $ with $p \geq 5$ and  $ \; H = {\mathbf A}_{p-1}$;

\medskip

\item  $ \; G = \mathbf{PSL}(2, 11) \; $  with $ \; p = 11 \; $ and $ \; H = {\mathbf A}_5 $;

\medskip

\item  $ \; G = {\mathbf M}_{11} \; $ with $ \; p = 11 \; $ and $ \; H = {\mathbf M}_{10} \; \; $ or $ \; G = {\mathbf M}_{23} \; $  with $ \; p = 23 \; $ and $ \; H = {\mathbf M}_{22}$;

\medskip

\item  $ \; G  = \mathbf{ PSL}(n,  \mathfrak{q}) \; $ with    $ \; H \; $  is the stabilizer of a point or a hyperplane of $ \; \mathbb{F}^{n}_{ \mathfrak{q}}. \; $  Then $ \;  \vert G : H \vert  = \displaystyle\frac{\mathfrak{q}^n - 1}{\mathfrak{q} - 1}  = p \; $;  $ \; n \; $ is a prime and $ \; \mathfrak{q} \; $ is a power of a prime number.
\end{enumerate}
\end{thm}

An immediate consequence of this result is the following.

\begin{cor}\label{corollary}
Let $ \; G \; $  be a simple non-abelian transitive  group of
prime degree $ \; p \; $  and $ \; H \leq G \; $ with $ \; \vert G : H \vert = p$. If there exists $ \; N \unlhd H \; $ with $ \; \vert H : N \vert  = q \; $ a prime number and $ \; q \ne p, \; $  then $ \; G \; $  is one of the following groups:
\begin{enumerate}
\item  $ \; G  = \mathbf{A}_5 \; $,  $p=5$ and $ \; q = 3$.\\
\item  $ \; G = {\mathbf M}_{11} \; $,  $p=11$ and $ \; q = 2$.\\
\item  $ \; G  = \mathbf{ PSL}(n,  \mathfrak{q}) \; $  with  $ \; p = \displaystyle\frac{\mathfrak{q}^n - 1}{\mathfrak{q} - 1} \; $  where $ \; n \; $  is prime and
\begin{itemize}
\item if $ \;  \mathfrak{q}  > 2,  \; \; $ then $ \; q \; $ is any prime  divisor of  $ \; \mathfrak{q} -1$;
\item if $ \; \mathfrak{q} = 2 \; \; $ then $ \; n = 3  \; $, $p=7$  and  $ \; q = 2$.
\end{itemize}
\end{enumerate}
\end{cor}
\begin{proof}
We check which of the groups listed in Theorem \ref{guralnick} satisfy the hypothesis of the corollary.

\begin{enumerate}
\item If $ \; G  = {\mathbf A}_p,  \; $ with $ \; p \geq 5, \; $ then   $ \; H = {\mathbf A}_{p-1}. \; $ Since for all $ \; p > 6 \; $ the alternating group $ \; H = {\mathbf A}_{p-1} \; $ is a simple group and $ \; \{1 \} \ne N \unlhd H$, we obtain $ \; p = 5$,  $H = {\mathbf A}_4 \; $ and $ \; N \; $ is the Sylow $2$-subgroup of $ \; H$. Hence $ \; q = 3.$

\medskip

\item It is clear that neither $ \;   G = \mathbf{PSL}(2, 11) \; $  and  $ \; H = {\mathbf A}_5 $  nor  $ \; G = {\mathbf M}_{23} \; $ and $ \; H = {\mathbf M}_{22} \; $  satisfy the hypothesis in the corollary.
\vspace{2mm}\\
If  $ \; G = {\mathbf M}_{11}, \; $then $ \; H = {\mathbf M}_{10}$. Since the only non-trivial normal subgroup of $ \; H \; $ is $ \; N = H^{\prime} = {\mathbf A}_6 \; $, we obtain $ \; q =  \vert H : N \vert = 2$.

\medskip

\item Let $ \; G  = \mathbf{ PSL}(n,  \mathfrak{q}) \; $ with $ \;  \vert G : H \vert  = \displaystyle\frac{\mathfrak{q}^n - 1}{\mathfrak{q} - 1}  = p. \; $ Then $ \; (n, \mathfrak{q} -1) = 1 \;$ and  $ \; G  = \mathbf{ PSL}(n,  \mathfrak{q})  =  \mathbf{SL}(n,  \mathfrak{q}).\; $
\vspace{2mm}\\
We know that
$$ H = \left\{\left( \begin{array}{cc} A & 0 \\
\omega & \kappa  \end{array}\right)  \; \; / \; \;  A \in \mathbf{ GL}(n-1, \mathfrak{q}), \; \; \kappa = \det(A^{-1}), \; \omega \in {\mathbb F}_{\mathfrak{q}}^{n-1}  \right\} .$$
Consider the normal subgroup of $ \; H \; $ given by
$$ M = \left\{\left( \begin{array}{cc} A & 0 \\
\omega & \kappa  \end{array}\right)  \; \; / \; \;  A \in \mathbf{ SL}(n-1, \mathfrak{q}), \; \; \kappa = \det(A^{-1}), \; \omega \in {\mathbb F}_{\mathfrak{q}}^{n-1}  \right\}.$$
Then $ \; M = H^{\prime} \; $ (the derived group of $H$) and  $ \; \vert H/M \vert = \mathfrak{q} - 1. \; $  Hence, for every $ \; N \unlhd H \; $ with $ \; \vert H : N \vert  = q \; $ a prime number we have that $ \; M  = H^{\prime} \leq N \; $ and $ \; q \; $ divides $ \;  \mathfrak{q} - 1.$
\vspace{2mm}\\
Therefore, for all prime numbers $ \; q \; $ dividing  $ \;  \mathfrak{q} - 1 \; $ there exists a normal subgroup $ \; N \unlhd H \; $ such that $ \; M = H^{\prime} \leq N \; $ and  $ \; \vert H : N \vert  = q. \; $
\vspace{2mm}\\
Finally, if   $ \;  \mathfrak{q} = 2, \; $ then $ \; \mathbf{GL}( n-1,  2) =  \mathbf{SL}( n-1,  2) =  \mathbf{PSL}( n-1,  2). \; $ In this case, if $ \; n > 3, \; $ then $ \; H^{\prime} = H \; $ and $ \; H \; $ has no normal subgroup of prime index. Hence $ \; n = 3  \; $ and  $ \; G = \mathbf{PSL}( 3,  2). \; $ Also $ \; p = 7, \; H = \mathbf {S}_4 \; $ and $ \; N = \mathbf{A}_4 \; $ with $ \; q = 2.$

\end{enumerate}

\end{proof}

Now we are able to prove the second part of the main Theorem.

\begin{thm}    	\label{T:last}
If $ \;  \mathcal{G} \; $ is a non-solvable group, then  $ \; {\mathcal U} \; $  is one of the groups listed in Theorem \ref{guralnick}.
\vspace{2mm}\\
Furthermore,
if $ \; \mathcal{K} = \{ 1 \} \; $   or $ \;  \mathcal{G}  = {\mathcal N} \: \mathcal{U}$    then $ \; \mathcal{U}  \; $ is one of the groups listed in Corollary \ref{corollary}.
\end{thm}
\begin{proof}
We know that  $ \;  \mathcal{G}= \mathcal{K}\: {\mathcal U} \; $ and   $ \;  \mathcal{G}/\mathcal{K} \cong {\mathcal U} \; $ is a simple  non-abelian transitive group of degree $ \; p. \; $ Hence, applying  Theorem \ref{guralnick} to the subgroup $ \; {\mathcal U} \; $ we obtain the result, where $ \; {\mathcal H} \cap {\mathcal U} \; $ is the corresponding subgroup of index $p$ in   $ \; {\mathcal U}.$
 \vspace{2mm}\\
If $ \;  \mathcal{K} = \{ 1 \} $, then  $ \;  \mathcal{G} = \mathcal{U} \; $  is a simple non-abelian transitive  group  of prime degree $ \; p. \; $ Hence, if ${\mathcal H}$ and $\mathcal{N}$ are the subgroups of $ \;  \mathcal{G} \; $  corresponding to $\mathcal{X}$ and $\mathcal{Y}$ respectively, then   $ \; \vert  \mathcal{G}  : \mathcal{H} \vert = p, \; \; {\mathcal N} \unlhd {\mathcal H} \; \; $ and $ \; \vert {\mathcal H} : {\mathcal N} \vert = q. \; $
So,  $ \; {\mathcal G} = \mathcal{U}  \; $ satisfies the hypotheses of  Corollary \ref{corollary}.
\vspace{2mm}\\
Finally,  observe that if $ \;  \mathcal{G} =  {\mathcal N}\:\mathcal{U}  \; $ then  $ \; \vert  \mathcal{G} : {\mathcal N} \vert = \vert \mathcal{U} : \mathcal{U} \cap {\mathcal N} \vert = p\, q$. Hence
	$$
	\vert \mathcal{G} : {\mathcal H} \vert = \vert \mathcal{U} : \mathcal{U} \cap {\mathcal H} \vert = p \; \; \; ; \ \; \vert \mathcal{U} \cap {\mathcal H} : \mathcal{U} \cap {\mathcal N} \vert = q \; \; \; \mbox{and}  \; \;  \;  \mathcal{U} \cap {\mathcal N} \unlhd  \mathcal{U} \cap {\mathcal H} .
	$$

	That is,  $ \; \mathcal{U}  \; $ satisfies the hypotheses of  Corollary \ref{corollary}.

\end{proof}

\begin{rem}
	Conversely, for each group listed in Corollary \ref{corollary} (and corresponding values of $p$ and $q$), there are corresponding covers $\mathcal{Y}\xrightarrow {\quad {\psi}} \mathcal{X} \xrightarrow {\quad \varphi} \mathbb{P}^{1}$ with $\psi$ a $q$-fold \'etale cover and $\varphi$ a totally ramified non-cyclic $p$-fold cover for which the monodromy group is that of the list.
	
	We illustrate with the examples in the next section.
	
\end{rem}

\section{Examples}

In what follows we consider the factorized cover
$$
\mathcal{Y}\xrightarrow {\quad {\psi}} \mathcal{X} \xrightarrow {\quad \varphi} \mathbb{P}^{1}
$$
with  $\mathcal{Y}\xrightarrow {\quad {\psi}} \mathcal{X}$  a degree $2$ \'etale cover and a totally ramified degree $7$ cover $\mathcal{X} \xrightarrow {\quad \varphi} \mathbb{P}^{1}$.  Then
$ \; \mathtt{g}_\mathcal{X} = 3 \; \; $ and $ \; \; \mathtt{g}_\mathcal{Y} = 5$.

Recall that $ \; \widetilde{\varphi} \; : \mathcal{Z} \to \mathbb{P}^{1}$ is the Galois closure of the composite cover $\varphi\circ\psi \; $, with $\mathcal{G}$  the corresponding Galois group,
 $\mathcal{N} \leq \mathcal{H} \leq \mathcal{G}$ the (conjugacy classes of) subgroups
such that $\mathcal{Z}/\mathcal{N} = \mathcal{Y}$ and $\mathcal{Z}/\mathcal{H} = \mathcal{X}$.

We have shown that then $\mathcal{G}  =\mathcal{K} \rtimes \mathcal{U}$,
where in this case $\mathcal{K} = \mathcal{H}_{\mathcal{G}}\cong \mathbb{Z}_2^s$ with $0 \leq s \leq 6$, and $\mathcal{U}$ is a simple transitive permutation group of degree $7$.

\medskip

We exhibit here the variety of possible Galois groups $\mathcal{G}$ that appear associated to this situation, even in the simplest case, when we  assume that the cover $\mathcal{X} \xrightarrow {\quad \varphi} \mathbb{P}^{1}$ has precisely three branch points. Equivalently, we assume that  $ \;  \mathcal{G}  \; $ acts on $ \; {\mathcal Z} \; $  with signature $ \; (0; 7,7,7) \; $; let $\{ a,b,(ab)^{-1}\}$ denote a set of geometric generators for this action.

\medskip

So consider $ \; \mathcal{G}  = \langle a, b \rangle  \leq \mathbf{A}_{14} \; $ with  $ \; a := (1,2,3,4,5,6,7)(8,9,10,11,12,13,14) \; $ and
\vspace{1mm}\\
\begin{enumerate}
	
	\item $ \; b:= (1,2,3,11,12,6,14)(4,5,13,7,8,9,10)$.

	Then $ \; \vert ab \vert = 7 \; \; ; \; \; \vert \mathcal{G} \vert = 2^3\cdot 7  \; \; $ and $ \; \mathcal{G} \cong \mathbb{Z}_2^3 \rtimes \mathbb{Z}_7$.

	Let $ \; \mathcal{H} = \langle h_1, h_2, h_3 \rangle  \leq \mathcal{G} \; $ and $ \; \mathcal{N} = \langle h_1, h_2 \rangle \leq \mathcal{H}$, where
	$$h_1 = (3,10)(5,12)(6,13)(7,14) \; \; ; \; \;
	h_2 = (1,8)(3,10)(4,11)(5,12) \; \; ; \; \;
	h_3 = (2,9)(3,10)(4,11)(7,14).
	$$
	
	Then $ \; \mathcal{H} \unlhd \mathcal{G} \; ;  \; \mathcal{H} = \mathbb{Z}_2^3 \; \; ; \; \; \vert \mathcal{H} : \mathcal{N} \vert = 2 $;  $ \; \mathcal{N}_{\mathcal{G}} = \{ 1 \}$, and  $  \; \mathtt{g}_ {\mathcal Z} = 17$.
	
	\medskip

	\item $ \; b:= (1,2,3,4,5,6,14)(7,8,9,10,11,12,13). \; $ 
	
	Then $ \; \vert ab \vert = 7 \; \; ; \; \; \vert \mathcal{G} \vert = 2^6\cdot 7 \; \; $ and $ \; \mathcal{G} \cong \mathbb{Z}_2^6 \rtimes \mathbb{Z}_7$.

	Let $ \; \mathcal{H} = \langle h_1, h_2, h_3,h_4,h_5,h_6  \rangle  \leq \mathcal{G} \; $ and $ \; \mathcal{N} = \langle h_1, h_2,h_3,h_4,h_5 \rangle \leq \mathcal{H} \; $ where
	$$h_1 = (1,8)(2,9)\; ; \;
	h_2 = (2,9)(3,10) \; ; \;
	h_3 = (3,10)(4,11)\;  ;  $$
	$$ h_4=(4,11)(5,12) \; ; \; h_5=(5,12)(6,13) \; ; \;  h_6 = (6,13)(7,14)
	$$
	
	Then $ \; \mathcal{H} \unlhd \mathcal{G} \; ;  \; \mathcal{H} = \mathbb{Z}_2^6 \; \; ; \; \; \vert \mathcal{H} : \mathcal{N} \vert = 2 \ $, $ \; \mathcal{N}_\mathcal{G} = \{ 1 \}$ and $  \;   \; \mathtt{g}_ {\mathcal Z} = 129$.
	
	\medskip
	
	Examples (1) and (2) correspond to the cases where the cover $\mathcal{X} \xrightarrow {\quad \varphi} \mathbb{P}^{1}$ is cyclic; that is, case i) in the Theorem. 
	
	\medskip
	
	\item $ \; b:= (1,2,3,4,6,14,5)(7,12,8,9,10,11,13)$.
	
	Then $ \; \vert ab \vert = 7 \; \; ; \; \; \vert \mathcal{G} \vert = 2^6\cdot \frac{7!}{2} \; \; $ and $ \; \mathcal{G} \cong \mathbb{Z}_2^6 \rtimes \mathbf{A}_7$.

	Let $ \; \mathcal{H} = \langle h_1, h_2, h_3  \rangle  \leq \mathcal{G} \; $ and $ \; \mathcal{N} = \langle n_1, n_2,n_3 \rangle \leq \mathcal{H} \; $ where
	$$h_1 = (1,2,3,4,5)(6,13)(7,14)(8,9,10,11,12) \; ; \;
	h_2 = (4,5,13)(6,11,12) \; ; \;
	h_3 = (2,9)(7,14) $$
	$$
	n_1 = (1,8)(2,9)(3,10)(4,11)\; \; ; \; \;
	n_2 = (1,5,4,3,2)(8,12,11,10,9) \; \; ; \; \;
	n_3 = (1,5,4,13,2)(6,9,8,12,11).
	$$
	
	Then $ \; \vert \mathcal{G} : \mathcal{H} \vert = 7 \; ;  \; \mathcal{H}_\mathcal{G} = \mathbb{Z}_2^6 \; \; ; \; \; \vert \mathcal{H} : \mathcal{N} \vert = 2 \ $ and $ \; \mathcal{N}_\mathcal{G} = \{ 1 \}$.
	
	Also, for $ \; \mathcal{U} = \langle u_1, u_2 \rangle \leq \mathcal{G} \; $ with
	$$ u_1=(1,2,3,4,5,13,7)(6,14,8,9,10,11,12) \; \; ; \; \;  u_2=(1,2,3,4,13,7,5)(6,14,12,8,9,10,11) $$
	we have $ \; \mathcal{U}  \cong \mathbf{A}_7$, and $  \;   \; \mathtt{g}_ {\mathcal Z} =  46081$.

     \medskip
     
    This example corresponds to case ii)a) in the Theorem.

	\medskip
	
	\item $ \; b:=(1,2,12,14,10,6,11)(3,13,4,8,9,5,7)$.
	
	Then $ \; \vert ab \vert = 7 \; \; ; \; \; \vert \mathcal{G} \vert = 2^3\cdot 168 \; \; $ and $ \; \mathcal{G} \cong \mathbb{Z}_2^3 \rtimes \mathbf{PSL}(3,2).$
	
	Let $ \; \mathcal{H} = \langle h_1, h_2, h_3, h_4  \rangle  \leq \mathcal{G} \; $ and $ \; \mathcal{N} = \langle n_1, n_2,n_3 \rangle \leq \mathcal{H} \; $ where
	$$h_1 = (2,7,3)(4,6,5)(9,14,10)(11,13,12) \; ; \;
	h_2 = (2,3,6,4)(5,7)(9,10,13,11)(12,14)\; ; \; $$
	$$h_3 = (3,10)(4,11)(5,12)(7,14) \; \; ; \; \; h_4= (1,8)(4,11)(5,12)(6,13);$$
	
	$$
	n_1=(1,8)(2,7)(4,11)(5,13)(6,12)(9,14); \;
	n_2=(2,7,3)(4,6,5)(9,14,10)(11,13,12); \;
	$$
	$$n_3=(2,13)(5,14)(6,9)(7,12).
	$$

	\medskip
	
	Then $ \; \vert \mathcal{G} : \mathcal{H} \vert = 7 \; ;  \; \mathcal{H}_\mathcal{G} = \mathbb{Z}_2^3 \; \; ; \; \; \vert \mathcal{H} : \mathcal{N} \vert = 2 \ $ and $ \; \mathcal{N}_\mathcal{G} = \{ 1 \}. $
	
	Also, for $ \mathcal{U} = \langle  a, u \rangle \leq \mathcal{G} \; $ with
	$  u =(1,2,5,7,3,6,4)(8,9,12,14,10,13,11) \; $
	we have $ \; \mathcal{U}  = \mathbf{PSL}(3,2)$ and $  \; \mathtt{g}_ {\mathcal Z} =  385$.
	
	Observe that $\mathcal{H}_\mathcal{G} \neq \{1\}$ and $\mathcal{N} \mathcal{U} = \mathcal{U} \mathcal{N}$, and therefore this example illustrates case ii)h) in the Theorem.
	
	\medskip

  \item $b=(1,2,12,14,10,6,11)(3,13,4,8,9,5,7)$.

Then $ \; \vert ab \vert = 7 \; \; ; \; \; \vert \mathcal{G} \vert = 2^3\cdot 168 \; \; $ and $ \; \mathcal{G} = \mathbb{Z}_2^3 \rtimes \mathbf{PSL}(3,2)$.

Let $ \; \mathcal{H}   = \langle h_1,h_2,h_3 \rangle \; $  and $ \; \mathcal{N} = \langle n_1, n_2,n_3 \rangle \leq \mathcal{H} \; $  with
$$h_1=(1,8)(2,9)(4,11)(7,14); \;  \;
h_2=(2,7,3)(4,6,5)(9,14,10)(11,13,12); 
$$
$$h_3=(2,3,6,4)(5,7)(9,10,13,11)(12,14)
$$
$$
n_1=(3,5)(4,7)(10,12)(11,14); \; \;
n_2=(2,10,5)(3,12,9)(4,14,13)(6,11,7);
$$
$$n_3=(2,6)(3,11)(4,10)(5,12)(7,14)(9,13)$$

and $ \; \mathcal{ U} = \langle a, u \rangle \; $ where
 \; $u=(1,2,5,7,3,6,4)(8,9,12,14,10,13,11). \; $

Then $ \; \vert  \mathcal{G} : \mathcal{H} \vert = 7   \; \; \; ; \; \; \mathcal{K} = \mathcal{H}_{ \mathcal{G}}  \cong {\mathbb{Z}_2}^{3}\ \; \;  ; \; \mathcal{U} \cong \mathbf{PSL}(3,2)$, and $\;   \; \mathtt{g}_ {\mathcal Z} = 385$.

Furthermore $ \; \vert \mathcal{H} : \mathcal{N} \vert = 2 \; \; ;  \; \;  \mathcal{N}_{ \mathcal{G}} = \{1\} \; \; $ and $ \;  \mathcal{U} \cap \mathcal{N} = \mathcal{U} \cap\mathcal{H} \cong \mathbf{S}_4$. Note that $ \; \mathcal{N} \; $ does not permute with $ \; \mathcal{U}$, and thus this example illustrates case ii)e) in the Theorem, with $s=3$.

   \medskip

	\item $ \; b=(1,2,5,7,10,13,4)(3,6,11,8,9,12,14)$.
	
	Then $ \; \vert ab \vert = 7 \; \; ; \; \; \vert \mathcal{G} \vert = 2^6\cdot 168 \; \; $ and $ \; \mathcal{G} = \mathbb{Z}_2^6 \rtimes \mathbf{PSL}(3,2).$

    Let $ \; \mathcal{H} = \langle h_1, h_2, h_3, h_4  \rangle  \leq \mathcal{G} \; $ and $ \; \mathcal{N} = \langle n_1, n_2,n_3, \rangle \leq \mathcal{H} \; $ where
	$$h_1 = (2,14,10)(3,9,7)(4,13,5)(6,12,11) \; ; \;
	h_2 = (2,10,6,11)(3,13,4,9)(5,14)(7,12)\; ; \; $$
	$$h_3 = (6,13)(7,14) \; \; ; \; \; h_4= (1,8)(4,11)(6,13)(7,14); \; \; n_1 = (2,11,9,4)(3,13,10,6);$$
	$$
	n_2 = (2,14,10)(3,9,7)(4,13,5)(6,12,11) \; \; ; \; \;
	n_3 =(3,4,10,11)(5,7)(6,13)(12,14).
	$$

	Then $ \; \vert \mathcal{G} : \mathcal{H} \vert = 7 \; ;  \; \mathcal{H}_\mathcal{G} = \mathbb{Z}_2^6 \; \; ; \; \; \vert \mathcal{H} : \mathcal{N} \vert = 2 \ $ and $ \; \mathcal{N}_\mathcal{G} = \{ 1 \}. $

	Also, for $ \;  \mathcal{U}  = \langle a, u \rangle \leq \mathcal{G} \; $ with
	$  u =(1,2,5,7,3,6,4)(8,9,12,14,10,13,11) \; $
	we have $ \;  \mathcal{U}  = \mathbf{PSL}(3,2)$ and $  \;   \; \mathtt{g}_ {\mathcal Z}=  3073$.
	
    \medskip
    
    This example illustrates case ii)e) in the Theorem with $s=6$.

    \medskip

	\item $ \; b:=  (1,2,5,7,3,6,4)(8,9,12,14,10,13,11)\; $ \vspace{2mm}\\
	Then $ \; \vert ab \vert = 7 \; \; ; \; \; \vert \mathcal{G} \vert = 168 \; \; $ and $ \; \mathcal{G} =   \mathbf{PSL}(3,2).$

	Let $ \; \mathcal{H} = \langle h_1, h_2 \rangle  \leq \mathcal{G} \; $ and $ \; \mathcal{N} = \langle n_1, n_2 \rangle \leq \mathcal{H} \; $ where
	$$h_1 = (1,2,6)(4,7,5)(8,9,13)(11,14,12) \; ; \;
	h_2 = (2,6)(3,7,4,5)(9,13)(10,14,11,12) $$
	$$ n_1 = (1,2,6)(4,7,5)(8,9,13)(11,14,12)\; \; ; \; \;
	n_2 = (3,5)(4,7)(10,12)(11,14)
	$$
	Then $ \; \vert \mathcal{G} : \mathcal{H} \vert = 7 \; ;  \; \mathcal{H}_\mathcal{G} = \{1\}$, so $\mathcal{G} =\mathcal{U}$,  $\vert \mathcal{H} : \mathcal{N} \vert = 2$, $ \mathcal{N}_\mathcal{G} = \{ 1 \}$ and $  \;   \; \mathtt{g}_ {\mathcal Z} = 49$.

    \medskip
    
    This example illustrates case ii)h) in the Theorem with $\mathcal{H}_\mathcal{G} = \{1\}$;, or, equivalently, with $s=0$.
\end{enumerate}


\begin{thebibliography}{99}


\bibitem{Allcock;Hall;2010}
D. Allcock and Ch. Hall.
\emph{Monodromy groups of {H}urwitz-type problems},
Adv. Math. \textbf{225} (2010), no.~1, 69--80.

\bibitem{Arenas;Rojas;2010}
L. Arenas-Carmona and A. M. Rojas.
\emph{Unramified prime covers of hyperelliptic curves and pairs of {$p$}-gonal curves},
In the tradition of {A}hlfors-{B}ers {V}, Contemp. Math., vol. 510, Amer. Math. Soc., Providence, RI, 2010, pp.~35--47.

\bibitem{Artebani;Pirola;2005}
M. Artebani and G. Pirola.
\emph{Algebraic functions with even monodromy},
Proc. Amer. Math. Soc. \textbf{133} (2005), no.~2, 331--341 (electronic).

\bibitem{Biggers;Fried;1986}
R. Biggers and M. Fried.
\emph{Irreducibility of moduli spaces of cyclic unramified covers of genus {$g$} curves},
Trans. Amer. Math. Soc. \textbf{295} (1986), no.~1, 59--70.

	
	\bibitem{Broughton;1991}
A. Broughton, \emph{Classifying finite group actions on surfaces of low
  genus}, J. Pure Appl. Algebra \textbf{69} (1991), no.~3, 233--270.

\bibitem{Carocca;Romero}
A. Carocca and M. Romero-Rojas.
\emph{On Galois group of factorized covers of curves},
Rev. Mat. Iberoam. \textbf{34} (2018), no. 4, 1853--1866.
	
\bibitem{CLR}
A. Carocca, H. Lange and R. E. Rodr\'{\i}guez.
\emph{\'Etale double covers of cyclic $p$-gonal covers},
J. Algebra {\textbf 538} (2019), 110 --126.
	
	
	
\bibitem{CHR}
A. Carocca, R. Hidalgo and R. E. Rodr\'{\i}guez.
{\emph  q-\'etale covers of cyclic $p$-gonal covers},
J. Algebra \textbf{573} (2021), 393–409. 

\bibitem{Diaz;Donagi;1989}
S. Diaz, R. Donagi, and D. Harbater.
\emph{Every curve is a {H}urwitz space},
Duke Math. J. \textbf{59} (1989), no.~3, 737--746.

\bibitem{Guralnick;1995}
R. M. Guralnick and M.  Neubauer.
\emph{Monodromy groups of branched coverings: the generic case},
Recent developments in the inverse {G}alois problem ({S}eattle, {WA}, 1993),
Contemp. Math., vol. 186, Amer. Math. Soc., Providence, RI, 1995, pp.~325--352.


\bibitem{GU} R. M. Guralnick, \emph{Subgroups of prime power index in a simple
group}, J. Algebra 81, (1983) 304-311


\bibitem{Guralnick;Shareshian;2007}
R. M. Guralnick and J. Shareshian.
\emph{Symmetric and alternating groups as monodromy groups of {R}iemann surfaces. {I}. {G}eneric covers and covers with many branch points},
Mem. Amer. Math. Soc. \textbf{189} (2007), no.~886, vi+128, With an appendix by Guralnick and R. Stafford.
	
\bibitem{Kanev;2006}
V. Kanev.
\emph{Hurwitz spaces of {G}alois coverings of {${\Bbb P}^1$}, whose {G}alois groups are {W}eyl groups},
J. Algebra \textbf{305} (2006), no.~1, 442--456.


\bibitem{Magaard;Volklein;2004}
K.  Magaard and H. V{\"o}lklein.
\emph{The monodromy group of a function on a general curve},
Israel J. Math. \textbf{141} (2004), 355--368.
	
\bibitem{m}
D. Mumford: {\it Prym varieties I}. in: Contr. to Analysis, Academic Press, 325-350 (1974).
	
\bibitem{sr}
S. Recillas, {\it Jacobians of curves with $ \; g_1^4$\' {} \!\!\!s are the Pryms of trigonal curves},  Bol. Soc. Mat. Mexicana (2) 19 (1974), no. 1, 9–13.
	
	
\bibitem{SRC} S. Reyes-Carocca, \emph{On pq-fold regular covers of the projective line},
To appear in Rev. R. Acad. Cienc. Exactas Fís. Nat. Ser. A Mat. RACSAM

\bibitem{Vetro;2007}
F.  Vetro.
\emph{Irreducibility of {H}urwitz spaces of coverings with monodromy groups {W}eyl groups of type {$W(B_d)$}}.
Boll. Unione Mat. Ital. Sez. B Artic. Ric. Mat. (8) \textbf{10} (2007), no.~2, 405--431.

\bibitem{Zariski;1978}
O.  Zariski.
\emph{Collected papers. {V}ol. {III}}, The MIT Press, Cambridge, Mass.-London, 1978, Topology of curves and surfaces, and special topics in the theory of algebraic varieties, Edited and with an introduction by M. Artin and B. Mazur, Mathematicians of Our Time.

\end{thebibliography}
\end{document}